\documentclass[11pt]{article}

\usepackage[utf8]{inputenc}
\usepackage[british]{babel}
\usepackage[a4paper, margin=2.5cm]{geometry}
\usepackage{amsmath, amsthm, amssymb, amsfonts}
\usepackage{mathtools}
\usepackage{mathabx}
\usepackage[backref]{hyperref}
\usepackage[abbrev, msc-links, backrefs]{amsrefs}
\usepackage[shortlabels, inline]{enumitem}
\usepackage[font={small, it}]{caption}
\usepackage{subcaption}
\usepackage{graphicx}
\usepackage{color}
\usepackage{xcolor}
\usepackage{thmtools}
\usepackage{bm}

\newcommand{\cE}{\ensuremath{\mathcal E}}

\newcommand{\cP}{\ensuremath{\mathcal P}}

\newcommand{\TINY}{\ensuremath{\mathrm{TINY}}}
\newcommand{\ATYP}{\ensuremath{\mathrm{ATYP}}}
\newcommand{\conn}{\ensuremath{\mathrm{conn}}}
\newcommand{\kconn}{\ensuremath{k\mathrm{-conn}}}
\newcommand{\Bin}{\ensuremath{\mathrm{Bin}}}

\newcommand{\eps}{\varepsilon}
\renewcommand{\phi}{\varphi}

\DeclareMathOperator*{\E}{\mathbb{E}}
\DeclareMathOperator*{\N}{\mathbb{N}}

\DeclareMathOperator*{\R}{\mathbb{R}}

\DeclarePairedDelimiter\ceil{\lceil}{\rceil}
\DeclarePairedDelimiter\floor{\lfloor}{\rfloor}

\let\setminus=\smallsetminus

\newcommand{\osref}[2]{%
  \setlength\abovedisplayskip{5pt plus 2pt minus 2pt}
  \setlength\abovedisplayshortskip{5pt plus 2pt minus 2pt}
  \ensuremath{\overset{\text{#1}}{#2}}
}

\newcommand{\Gnp}{G_{n, p}}
\newcommand{\Xa}{X_{\mathrm{atyp}}}
\newcommand{\Xt}{X_{\mathrm{typ}}}

\newcommand{\binoms}[2]{{\textstyle\binom{#1}{#2}}} % small binomial

\definecolor{royalazure}{rgb}{0.0, 0.22, 0.66}
\definecolor{forestgreen}{rgb}{0.13, 0.55, 0.13}

\declaretheorem[parent=section]{theorem}
\declaretheorem[sibling=theorem]{lemma}
\declaretheorem[sibling=theorem]{proposition}
\declaretheorem[sibling=theorem]{claim}
\declaretheorem[sibling=theorem]{corollary}
\declaretheorem[sibling=theorem,style=definition]{definition}

\setlist{itemsep=0.1em, topsep=0.1em, parsep=0.1em, partopsep=0.1em}

\hypersetup{
  colorlinks,
  linkcolor={red!60!black},
  citecolor={green!50!black},
  urlcolor={blue!80!black}
}

\title{On resilience of connectivity in the evolution of random graphs}
\author{}

\makeatletter
\renewcommand\@date{{%
  \vspace{-\baselineskip}%
  \large\centering
  \begin{tabular}{@{}c@{}}
    Luc Haller \\
    \normalsize \texttt{hallerl@student.ethz.ch}
  \end{tabular}%
  \qquad \qquad
  \begin{tabular}{@{}c@{}}
    Milo\v{s} Truji\'{c}\thanks{author was supported by grant no.\ 200021 169242
    of the Swiss National Science Foundation.} \\
    \normalsize \texttt{mtrujic@inf.ethz.ch}
  \end{tabular}

  \bigskip

  Institute of Theoretical Computer Science \\
  ETH Z\"{u}rich, 8092 Z\"{u}rich, Switzerland

  \bigskip

  \date{}
}}
\makeatother

\begin{document}
\maketitle

\begin{abstract}
  In this note we establish a resilience version of the classical hitting time
  result of Bollob\'{a}s and Thomason regarding connectivity. A graph $G$ is
  said to be $\alpha$-resilient with respect to a monotone increasing graph
  property $\cP$ if for every spanning subgraph $H \subseteq G$ satisfying
  $\deg_H(v) \leq \alpha \deg_G(v)$ for all $v \in V(G)$, the graph $G - H$
  still possesses $\cP$. Let $\{G_i\}$ be the random graph process, that is a
  process where, starting with an empty graph on $n$ vertices $G_0$, in each
  step $i \geq 1$ an edge $e$ is chosen uniformly at random among the missing
  ones and added to the graph $G_{i - 1}$. We show that the random graph process
  is almost surely such that starting from $m \geq (\tfrac{1}{6} + o(1)) n \log
  n$, the largest connected component of $G_m$ is $(\tfrac{1}{2} -
  o(1))$-resilient with respect to connectivity. The result is optimal in the
  sense that the constants $1/6$ in the number of edges and $1/2$ in the
  resilience cannot be improved upon. We obtain similar results for
  $k$-connectivity.
\end{abstract}

%!TEX root = connectivity_resilience.tex
\section{Introduction}\label{sec:introduction}

Random graph theory dates back to 1959 and two seminal papers of Erd\H{o}s and
R\'{e}nyi~\cite{erdds1959random} and Gilbert~\cite{gilbert1959random}. By now it
is a well-studied research area with applications in different fields. A more
recent trend, started with a paper of Sudakov and Vu~\cite{sudakov2008local}, is
to study the {\em resilience} of certain properties in random graphs, which has
since attracted considerable attention (e.g.~\cites{allen2016bandwidth,
  balogh2011local, balogh2012corradi, bottcher2013almost, lee2012dirac,
vskoric2018local} and recent surveys~\cites{bottcher2017large,
sudakov2017robustness}). In other words, knowing that a random graph behaves in
a certain way, the question then becomes how robust that behaviour is with
respect to modifications (e.g.\ deletion of the edges). Two of the most
prominent examples of robustness of graphs can be seen in the results of
Tur\'{a}n~\cite{turan1941extremalaufgabe} and Dirac~\cite{dirac1952some}.
Tur\'{a}n's theorem determines how many edges one has to remove from a complete
graph on $n$ vertices $K_n$ before it becomes $K_r$-free ({\em global
resilience}), while Dirac's theorem measures how many edges touching each vertex
one is allowed to remove before the graph ceases to be Hamiltonian ({\em local
resilience}). The removal of all the edges incident to a vertex with minimum
degree prevents the containment of any spanning structure (e.g.\ connectivity,
Hamiltonicity, perfect matchings, etc.) giving a trivial upper bound on the
global resilience in these cases. In order to study robustness of such
properties it is natural to turn our attention to the local resilience.

In a more formal setting, local resilience of a graph $G$ with respect to a
monotone increasing graph property $\cP$ is defined as follows:

\begin{definition}[$\alpha$-resilience]
  \label{def:local-resilience}
  Let $G = (V, E)$ be a graph, $\cP$ a monotone increasing graph property, and
  $\alpha \in [0, 1]$ a constant. We say that $G$ is $\alpha$-resilient with
  respect to $\cP$ if for every spanning subgraph $H \subseteq G$ satisfying
  $\deg_H(v) \leq \alpha \deg_G(v)$ for all $v \in V$, we have $G - H \in \cP$.
\end{definition}

Generally, being $\alpha$-resilient means that an adversary cannot destroy
property $\cP$ by removing an arbitrary $\alpha$-fraction of the edges incident
to every vertex. In light of this, Dirac's theorem states that the complete
graph $K_n$ is $(1/2)$-resilient with respect to Hamiltonicity. On the other
hand, allowing the adversary to remove a bit more than a $(1/2)$-fraction of the
edges incident to each vertex proves to be enough in order to destroy all
Hamilton cycles, and even disconnect the graph. There is a vast number of
important results in extremal combinatorics which study `resilience' of the
complete graph and we refer the interested reader to
e.g.~\cites{bottcher2009proof, hajnal1970proof, komlos2001proof, komlos1998posa,
kouider1996covering}.

In this paper we show that one of the first results in random graph theory holds
in a resilient fashion. A fundamental result of Erd\H{o}s and
R\'{e}nyi~\cite{erdds1959random} shows that if $m = \tfrac{n}{2} (\log n +
c_n)$, then {\em with high probability}\footnote{We say that an event holds with
high probability, w.h.p.\ for short, if the probability of it to hold tends to
$1$ as $n \to \infty$.} $G_{n, m}$, a graph drawn uniformly at random among all
graphs with $n$ vertices and $m$ edges, is connected if $c_n \to \infty$. Its
strongest `hitting time' version was shown much later by Bollob\'{a}s and
Thomason~\cite{bollobas1985random} and is the one which we consider here (see
also Bollob\'{a}s~\cite{bollobas1998random}*{Chapter 7}). For an integer $n \in
\N$ we let $\{G_i\}$ define the {\em random graph process} as follows. Let $G_0$
be an empty graph on $n$ vertices; at every step $i \in \{1, \ldots,
\binom{n}{2}\}$ let $G_i$ be obtained from $G_{i - 1}$ by choosing an edge $e
\notin G_{i - 1}$ uniformly at random and adding it to $G_i$. This defines a
sequence of nested graphs $\{G_i\}_{i = 0}^{N}$, where $N = \binom{n}{2}$, $G_0$
is an empty graph, and $G_N$ a complete graph on $n$ vertices.

A trivial necessary condition for a graph to be connected is that it has minimum
degree at least one. It turns out that in the random graph process this is also
sufficient. In other words, the edge $e_i$ that makes the last isolated vertex
disappear w.h.p.\ also makes the graph $G_i$ connected. Put into a more formal
setting, for a monotone increasing graph property $\cP$ and a random graph
process $\{G_i\}$, we define the {\em hitting time} with respect to $\cP$, and
write $\tau_\cP$, as:
\[
  \tau_{\cP} = \min{\{m \geq 0 \colon G_m \in \cP \}}.
\]
One of the results of Bollob\'{a}s and Thomason~\cite{bollobas1985random} proves
that w.h.p.\ $\tau_1 = \tau_\conn$, where $\tau_1$ denotes the hitting times of
`having minimum degree at least one' and $\tau_\conn$ that of `connectivity'.

The hitting time statements about the random graph processes are in some sense
the most precise results concerning random graphs one could hope for.
Unsurprisingly, they are also the most difficult to analyse and only a handful
of statements are known regarding some basic graph properties (such as
connectivity, perfect matchings, Hamiltonicity, etc.).

Our contribution is to show that almost surely not only the graph is connected
at the point $\tau_1$, but it is resiliently connected, that is it stays
connected even after the adversary removes at most a $(1/2 - o(1))$-fraction of
the edges touching each vertex. This continues the line of research recently
initiated by Nenadov, Steger, and the second
author~\cite{nenadov2018resilience}, and independently
Montgomery~\cite{montgomery2017}, of studying resilience of properties in random
graph processes. We further obtain similar results with respect to
$k$-connectivity.

\subsection{Our results}

Our primary objective is to show the resilience of connectivity in the random
graph process. Instead of dealing with the random graph process $\{G_i\}$
directly, one may find it more convenient to think about the Erd\H{o}s-R\'{e}nyi
random graph $G_{n, m}$. It is a well-known fact that for every $m \in \{1,
\ldots, \binoms{n}{2}\}$ the graph $G_m$ has the same distribution as $G_{n,
m}$. Moreover, as long as $m$ is above a certain value (which is the case for
us), standard connections between the model $G_{n, m}$ and the {\em binomial
random graph} model {$\Gnp$}\footnote{For an integer $n$ and a function $0 < p =
p(n) < 1$ we denote by $\Gnp$ the probability space of graphs on $n$ vertices
where each edge is present with probability $p$ independently of other edges.}
lead to direct consequences for this model as well.

Our first result shows that as long as $m$ is not too small, the largest
connected component of $G_{n, m}$, which we refer to as the {\em giant}, is
resilient with respect to connectivity.

\begin{restatable}{theorem}{conntheorem}\label{thm:main-thm-conn}
  Let $\eps > 0$ be a constant and consider the random graph process $\{G_i\}$.
  Then w.h.p.\ for every $m \geq \frac{1 + \eps}{6} n\log n$ we have that the
  giant of $G_m$ is $(1/2 - \eps)$-resilient with respect to connectivity.
\end{restatable}

The value of $m$ is asymptotically optimal. Indeed, having $m = \frac{1 -
\eps}{6} n\log n$ is enough for the existence of {\em cherries} `attached' to
the giant, i.e.\ two vertices of degree one with a common neighbour of degree
three, which the adversary can easily disconnect by removing the edge from the
degree three vertex which is not incident to a degree one vertex
(see~\cite{luczak1991tree}). Additionally, the constant $1/2$ is optimal as we
show in Proposition~\ref{prop:optimal-constant} at the end of
Section~\ref{sec:proof}.

The previous theorem immediately implies a hitting time result for the
resilience of connectivity.

\begin{theorem}\label{thm:hitting-time-conn}
  Let $\eps > 0$ be a constant. Consider the random graph process $\{G_i\}$ and
  let $\tau_1 = \min{\{m \colon \delta(G_m) \geq 1\}}$ denote the step in which
  the last isolated vertex disappears. Then w.h.p.\ we have that $G_{\tau_1}$ is
  $(1/2 - \eps)$-resilient with respect to connectivity.
\end{theorem}

We now turn our attention to $k$-connectivity. Recall, a graph is said to be
$k$-connected if the removal of at most $k - 1$ vertices does not disconnect it.
For the random graph process $\{G_i\}$, let us define
\[
  \tau_k = \min{\{ m \colon \delta(G_m) \geq k \}} \quad \text{and} \quad
  \tau_{\kconn} = \min{\{ m \colon G_m \text{ is $k$-connected} \}},
\]
denoting the point at which $\delta(G_m) \geq k$ and at which $G_m$ becomes
$k$-connected, respectively. Erd\H{o}s and R\'{e}nyi~\cite{erdHos1964strength}
were the first to show that for $m = \tfrac{n}{2}(\log n + (k - 1)\log\log n +
c_n)$ w.h.p.\ $G_{n, m}$ is $k$-connected if $c_n \to \infty$, which also
coincides with the threshold for $G_{n, m}$ to have minimum degree $k$. This
result was later strengthened by Bollob\'{a}s and
Thomason~\cite{bollobas1985random} who proved the hitting time statement, that
is w.h.p.\ $\tau_k = \tau_{\kconn}$. Therefore, the trivial necessary condition
for $k$-connectivity---having minimum degree at least $k$---turns out to be
sufficient as well.

In order to state our second result regarding $k$-connectivity, we need a
slightly different notion of resilience. In the case of $1$-connectivity (or
simply connectivity) and $2$-connectivity, $(1/2 - o(1))$-resilience is not
enough for the adversary to make the minimum degree of the graph drop below one
and two, respectively. However, as soon as $k \geq 3$, the adversary would be
allowed to remove an edge incident to a degree three vertex, and could easily
prevent $3$-connectivity. We go around this fact by slightly restricting the
power of the adversary with the following definition.

\begin{definition}[$(\alpha, k)$-resilience]\label{def:strange-resilience}
  Let $G = (V, E)$ be a graph, $\cP$ a monotone increasing graph property,
  $\alpha \in [0, 1]$ a constant, and $k \geq 1$ an integer. We say that $G$ is
  $(\alpha, k)$-{\em resilient} with respect to $\cP$ if for every spanning
  subgraph $H \subseteq G$ such that $\deg_H(v) \leq \alpha \deg_G(v)$ and
  $\deg_{G - H}(v) \geq k$ for all $v \in V$, we have $G - H \in \cP$.
\end{definition}

As observed above, the main obstruction for not having $k$-connectivity in the
beginning of the process is the existence of vertices with degree smaller than
$k$. However, something can still be said for such sparse(r) graphs. For an
integer $k \geq 2$ and a graph $G$ we define the $k$-{\em core} of $G$ to be the
(possibly empty) graph obtained by successively removing vertices of degree less
than $k$ from $G$. With this in mind we are ready to state our second result.

\begin{restatable}{theorem}{kconntheorem}\label{thm:main-thm-k-conn}
  Let $k \geq 2$ be an integer, $\eps > 0$ a constant, and consider the random
  graph process $\{G_i\}$. Then w.h.p.\ for every $m \geq \frac{1 + \eps}{6} n
  \log n$ we have that the $k$-core of $G_m$ is $(1/2 - \eps, k)$-resilient with
  respect to $k$-connectivity.
\end{restatable}

Even though the value $m \geq n\log n/6$ is optimal for connectivity, it is
conceivable that the $k$-core is resilient with respect to $k$-connectivity much
earlier in the process, that is as soon as it is $k$-connected, which happens
roughly for $m = cn$, for some constant $c$ depending on $k$, as shown by
Bollob\'{a}s~\cite{bollobas84evolutionhitting}. From a recent result of
Montgomery~\cite{montgomery2017} one deduces that the $k$-core is $(1/2 -
o(1))$-resilient with respect to $2$-connectivity, but only if the constant $k$
is much larger than two, and furthermore, no conclusions can be drawn for
$k$-connectivity. We leave this as a question for further research.

Similarly as before we immediately obtain a hitting time version of the result.

\begin{theorem}\label{thm:hitting-time-k-conn}
  Let $k \geq 2$ be an integer and $\eps > 0$ a constant. Consider the random
  graph process $\{G_i\}$ and let $\tau_k = \min{\{m \colon \delta(G_m) \geq
  k\}}$ denote the step in which the last vertex of degree at most $k - 1$
  disappears. Then w.h.p.\ we have that $G_{\tau_k}$ is $(1/2 - \eps,
  k)$-resilient with respect to $k$-connectivity.
\end{theorem}

As a direct corollary we also get a statement about the classic notion of $(1/2
- o(1))$-resilience for $k$-connectivity.

\begin{corollary}
  Let $k \geq 2$ be an integer and $\eps > 0$ a constant. Consider the random
  graph process $\{G_i\}$ and let $\tau_{2k - 2} = \min{\{ m \colon \delta(G_m)
  \geq 2k - 2 \}}$ denote the step in which the last vertex of degree at most
  $2k - 3$ disappears. Then w.h.p.\ we have that $G_{\tau_{2k - 2}}$ is $(1/2 -
  \eps)$-resilient with respect to $k$-connectivity.
\end{corollary}

%!TEX root = connectivity_resilience.tex
\section{Preliminaries}\label{sec:preliminaries}

Our graph theoretic notation follows standard textbooks (see,
e.g.~\cite{bondy2008graphtheory}). In particular, given a graph $G$ we denote by
$V(G)$ and $E(G)$ the set of its vertices and edges, respectively, and by $v(G)$
and $e(G)$ their sizes. For subsets of vertices $X, Y \subseteq V(G)$, $G[X]$
stands for the subgraph of $G$ induced by $X$, $G[X, Y]$ for the bipartite
subgraph with bipartition $(X, Y)$, and $E_G(X, Y)$ denotes the set of edges
between $X$ and $Y$ in $G$, i.e.\ $E_G(X, Y) := \{\{v, w\} \colon v \in X, w \in
Y,\ \text{and}\ \{v, w\} \in E(G)\}$ and $e_G(X, Y)$ denotes its size. For
short, we write $E_G(X) := E_G(X, X)$ and $e_G(X)$ for its size. Furthermore, we
write $N_G(X) := \{v \in V(G) \colon \exists u \in X\ \text{such that}\ \{u, v\}
\in E(G)\}$ for the neighbourhood of $X$ in $G$. Given a vertex $v \in V(G)$ we
abbreviate $N_G(\{v\})$ to $N_G(v)$ and let $\deg_G(v)$ be the size of its
neighbourhood, that is the degree of $v$ in $G$. We use $\delta(G)$ for the
minimum degree of $G$. For $\ell \in \N$ and a vertex $v \in V(G)$, we define
the $\ell$-neighbourhood of $v$ as the set of all vertices which lie at distance
at most $\ell$ from $v$, and write $N_{G}^{\ell}(v)$, excluding $v$ itself. We
omit the subscript $G$ whenever it is clear from the context which graph we
refer to.

For $x, y, \eps \in \R$, we write $x \in (1 \pm \eps)y$ to denote $(1 - \eps)y
\leq x \leq (1 + \eps)y$. Throughout, we use the natural logarithm $\log x =
\log_e x$. Ceilings and floors are omitted whenever they are not essential. We
make use of the standard asymptotic notation $o, \omega, O$, and $\Omega$.
Lastly, we use subscripts with constants such as $C_{3.1}$ to indicate that
$C_{3.1}$ is a constant with the properties as in the statement of
Claim/Lemma/Proposition/Theorem 3.1.

We make use of the standard estimate for deviation of a binomially distributed
random variable $\Bin(n, p)$ with parameters $n$ and $p$ from its mean (see,
e.g.~\cite{frieze2015introduction}).

\begin{lemma}[Chernoff's inequality]\label{lem:chernoff}
  Let $X \sim \Bin(n, p)$ and let $\mu := \E[X]$. Then for all $\delta \in (0,
  1)$:
  \begin{itemize}
    \item $\Pr[X \geq (1 + \delta) \mu] \leq e^{-\frac{\delta^2 \mu}{3}}$, and
    \item $\Pr[X \leq (1 - \delta) \mu] \leq e^{-\frac{\delta^2 \mu}{2}}$.
  \end{itemize}
\end{lemma}

Although our main results concern the random graph process, in the proof we
heavily rely on the properties of the binomial random graph $\Gnp$. Next is a
bound on the number of edges certain subsets can have in a random graph
$\Gnp$ (see~\cite{krivelevich2006pseudo}*{Corollary~2.3}).

\begin{lemma}\label{lem:gnp-num-edges}
  Let $p = p(n) \leq 0.99$. Then w.h.p.\ $G \sim \Gnp$ is such that every set $X
  \subseteq V$ satisfies
  \[
    \binom{|X|}{2}p - c|X| \sqrt{np} \leq e(X) \leq \binom{|X|}{2}p + c|X|
    \sqrt{np},
  \]
  for some absolute constant $c > 0$.
\end{lemma}

The biggest obstacle in working with (very) sparse random graphs, that is for
values of $p$ close to the threshold of connectivity, is that one cannot
guarantee that w.h.p.\ even the degree of every vertex is concentrated around
its expectation. Hence, the vertices whose degrees deviate significantly from
the average require some special care. We introduce two classes of such vertices
in the following definition.

\begin{definition}
  For $\delta, p \in [0, 1]$ and a graph $G$ with $n$ vertices, we define
  \begin{align*}
    \TINY_{p, \delta}(G) &= \{v \in V(G) \colon \deg_G(v) < \delta np\}, \\
    \ATYP_{p, \delta}(G) &= \{v \in V(G) \colon \deg_G(v) \notin (1 \pm
    \delta)np\}.
  \end{align*}
  We refer to the vertices in $\TINY_{p, \delta}(G)$ as {\em tiny} and to the
  vertices in $\ATYP_{p, \delta}(G)$ as {\em atypical}.
\end{definition}

Note that if $p$ in the above definition is roughly such that $np$ is the
average degree of a vertex in $G$, then one can think of $\TINY$ as a set of
vertices whose degree is much smaller than the average, and $\ATYP$ as a set of
vertices whose degree is even slightly away from the average. Let us point out
that if $p \geq (1 + \eps)\log n/n$ and $\delta$ is small enough (depending on
$\eps$), then w.h.p.\ $G \sim \Gnp$ contains no tiny vertices, and similarly if
$p \geq C\log n/n$ for large enough $C$ (depending on $\delta$), then w.h.p.\ $G
\sim \Gnp$ contains no atypical vertices.

As their name indicates, the atypical vertices are `rare' and do not occupy a
significant fraction of the graph.

\begin{lemma}\label{lem:atypical-is-small}
  For every $\delta > 0$, if $p \geq \log n/(3n)$, then $G \sim \Gnp$ w.h.p.\
  satisfies:
  \[
    |\ATYP_{p, \delta}(G)| \leq n/\log n.
  \]
\end{lemma}
\begin{proof}
  For a fixed vertex $v \in V(G)$ by Chernoff's inequality we have
  \[
    \Pr[\deg_G(v) \notin (1 \pm \delta) np] \leq 2 e^{-\delta^2 np/3} \leq
    n^{-\gamma},
  \]
  for some $\gamma > 0$. Therefore, the expected size of the set $\ATYP_{p,
  \delta}(G)$ is $n^{1 - \gamma}$ and Markov's inequality shows that w.h.p.\
  $|\ATYP_{p, \delta}(G)| \leq n / \log n$.
\end{proof}

We briefly discuss the fact that with the previous definition at hand we could
use a variant of \cite{nenadov2018resilience}*{Definition 2.4} in order to show
marginally stronger statements than the ones from
Theorem~\ref{thm:hitting-time-conn} and Theorem~\ref{thm:hitting-time-k-conn}.

\begin{definition}{\cite{nenadov2018resilience}*{Definition 2.4}}
  Given a graph $G = (V, E)$, a graph property $\cP$, constants $\alpha,
  \delta_t, \delta_a \in [0, 1]$, and integers $K_t, K_a \in \N$, we say that
  $G$ is $(\alpha, \delta_t, K_t, \delta_a, K_a)$-{\em resilient} with respect
  to $\cP$ if for every spanning subgraph $H \subseteq G$ such that
  \[
    \deg_H(v) \leq
      \begin{cases}
        \deg_G(v) - K_t, & \text{if } v \in \TINY_{p, \delta_t}(G), \\
        \deg_G(v) - K_a, & \text{if } v \in \ATYP_{p, \delta_a}(G) \setminus
        \TINY_{p, \delta_t}(G), \\
        \alpha\deg_G(v), & \text{otherwise},
     \end{cases}
  \]
  for every $v \in V$, where $p = |E| / \binom{|V|}{2}$, we have $G - H \in
  \cP$.
\end{definition}

However, just using $(1/2 - o(1), \delta_t, 1, \delta_a, K_a)$-resilience for
connectivity would not suffice for the following reason: for $m \leq (1/4 -
\eps)n\log n$ w.h.p.\ there exists a path of length two `attached' to the giant
of $G_m$, which then an adversary could disconnect by removing the edge that
`attaches' it to the giant. Note that the removal of such an edge is not
possible under the definition of $(1/2 - o(1))$-resilience and hence we would
need a slight modification in the definition above. Lastly, $(1/2 - o(1),
\delta_t, k, \delta_a, K_a)$-resilience would suffice for $k$-connectivity for
every $k \geq 2$. In conclusion, we believe that such a `complication' would
greatly reduce the readability of our paper, unnecessarily distract the reader
from the main point, and is as such not worth pursuing.

%!TEX root = connectivity_resilience.tex
\section{The proof}\label{sec:proof}

Recall that we are trying to prove that the random graph process $\{G_i\}$ is
typically such that starting from $m \geq \frac{1 + \eps}{6} n \log n$ the giant
of $G_m$ is resilient with respect to connectivity. Before diving into the
details we give a brief outline of the ideas used.

We follow the exact path paved by the proof of
\cite{nenadov2018resilience}*{Proposition 3.1}. The most natural thing one could
try is to show that for a fixed $m \geq n \log n/6$ the giant of $G_m$ is
w.h.p.\ resilient with respect to connectivity and apply a union bound over all
such $m$. Unfortunately, this approach would fail as this probability is roughly
$1 - e^{- \alpha \cdot 2m/n}$, for a small constant $\alpha > 0$---clearly not
enough for a union bound over all values of $m$ of order $n \log n$. Instead, we
group graphs into batches of size $\eps n \log n$ and show that w.h.p.\ all
graphs in a batch satisfy the property {\em simultaneously}. This allows us to
apply a union bound over only a constant number of batches in order to cover all
values of $m$ up to $C n \log n$.  Choosing $C$ to be large enough, we may then
cover the remaining values of $m$ individually, as now for each fixed one the
statement holds with probability at least $1 - o(n^{-3})$.

For a cleaner exposition, we generate the graph $G_m$ with the help of a
binomial random graph $\Gnp$, instead of doing it from scratch. Namely, we
sample the graphs $G^{-} \sim G_{n, p_0}$ and $G_{n, p'}$ where the values of
$p_0$ and $p'$ are such that $G^{-}$ w.h.p.\ has at most some $m_0$ edges and
$G^{+} = G^{-} \cup G_{n, p'}$ has at least $(1 + \eps/6)m_0$ edges. Taking now
a permutation $\pi$ of the edges of $G_{n, p'}$ uniformly at random we may
generate each $G_m$ as a union of $G^{-}$ and the first $m - e(G^{-})$ edges
given by $\pi$.

Crucially, the properties of the graphs $G^-$ and $G^+$ that we make use of are
such that all graphs `squeezed' in between them also satisfy them. The most
important one concerns tiny and atypical vertices and states that not many of
them can be `clumped' together. Actually, an even stronger statement is true:
adding an additional $\eps$-fraction of random edges to $G^-$ in order to obtain
$G^+$ does not make tiny and atypical vertices of $G^-$ more clumped. The next
lemma, which is a special case of \cite{nenadov2018resilience}*{Lemma~2.6},
captures this precisely.

\begin{lemma}\label{lem:neighbourhood-atyp-tiny}
  For every $\eps > 0$ there exist positive constants $\delta(\eps)$ and
  $L(\eps)$ such that if $p_0 \geq (1 + \eps)\log n/(3n)$ and $p' \leq \eps p_0$
  then w.h.p.\ the following holds. Let $G^- \sim G_{n, p_0}$ and set $G^+ = G^-
  \cup G_{n, p'}$ and $p_1 = 1 - (1 - p_0)(1 - p')$. Then:
  \begin{enumerate}[(i)]
    \item\label{neighbourhood-tiny} for every $v \in V(G^+)$ we have
      $|N_{G^+}^3(v) \cap \TINY_{p_0, \delta}(G^-)| \leq 2$,
    \item\label{neighbourhood-atyp} for every $v \in V(G^+)$ we have
      $|N_{G^+}(v) \cap \ATYP_{p_0, \delta}(G^-)| \leq L$.
    \item\label{cycle-tiny} for every cycle $C \subseteq G^+$ with $v(C) = 3$,
      we have $|V(C) \cap \TINY_{p_0, \delta}(G^-)| \leq 1$.
  \end{enumerate}
\end{lemma}

Lastly, in order for the whole strategy to work in the resilience setting we
keep only the edges of $G_m$ existing in $G^-$, but allow the adversary to
remove edges with respect to the degrees in $G_m$.

\begin{proposition}\label{prop:graph-process}
  Let $k \geq 1$ be an integer. For every $\eps > 0$ and integer $m_0 \geq
  \frac{1 + \eps}{6} n\log n$ the random graph process $\{G_i\}$ w.h.p.\ has the
  following property: for every integer $m_0 \leq m \leq (1 + \eps/6) m_0$ the
  giant of $G_m$ is $(1/2 - \eps)$-resilient with respect to connectivity.
  Furthermore, if $k \geq 2$, the $k$-core of $G_m$ is $(1/2 - \eps,
  k)$-resilient with respect to $k$-connectivity.
\end{proposition}
\begin{proof}
  Given $\eps$ let us define $\delta = \min{\{\eps/4,
  \delta_{\ref{lem:neighbourhood-atyp-tiny}}(\eps/2)\}}$ and $L =
  L_{\ref{lem:neighbourhood-atyp-tiny}}(\eps/2)$, and let $c =
  c_{\ref{lem:gnp-num-edges}}$. Take $p_0 = (1 - \eps/16) m_0 / \binom{n}{2}$,
  $p' = (\eps/2) p_0$, and let $G^+$ be the union of two independent copies of
  random graphs $G^- \sim G_{n, p_0}$ and $G_{n, p'}$. Then $G^+$ is distributed
  as $G_{n, p_1}$, where $p_1 = 1 - (1 - p_0)(1 - p')$. By
  Lemma~\ref{lem:gnp-num-edges} (or simply Chernoff's inequality) we have that
  the number of edges in $G^-$ is w.h.p.\ at most $m_0$ and the number of edges
  in $G^+$ is w.h.p.\ at least $(1 + \eps/6) m_0$.

  For the rest of the proof consider some $m_0 \leq m \leq (1 + \eps/6) m_0$. As
  the proof for both connectivity and $k$-connectivity (for $k \geq 2$) is
  identical, we treat them together and think of the graph we work with as the
  giant of $G_m$ in the former and as the $k$-core of $G_m$ in the latter. With
  this in mind, let $G \subseteq G_m$ be the giant of $G_m$, and respectively a
  subgraph obtained by iteratively removing all vertices of degree less that $k$
  from $G_m$ for $k \geq 2$. Let $V = V(G)$ denote the vertex set of $G$, and
  let $\TINY$ and $\ATYP$ be sets of vertices defined as:
  \[
    \TINY := \TINY_{p_0, \delta}(G^-) \cap V \quad \text{and} \quad \ATYP :=
    \ATYP_{p_0, \delta}(G^-) \cap V.
  \]

  It is a well-known fact that for $m \geq n\log n/6$ the size of the giant
  (resp.\ the $k$-core) of $G_m$ is at least $(1 - o(1))n$
  (cf.~\cites{frieze2015introduction, luczak1989sparse, luczak1991size}). If we
  show that for every graph $H$ whose vertex degrees fulfil
  \begin{align}\label{eqn:H-degree-bounds}
    \deg_H(v) \leq \min{\{ (1/2 - \eps)\deg_G(v),\ \deg_G(v) - k \}}
  \end{align}
  the graph $G - H$ is $k$-connected, then $G$ is $(1/2 - \eps, k)$-resilient
  with respect to $k$-connectivity.

  First, we list a series of properties that the graph $G$ satisfies, which we
  subsequently show are sufficient for proving the $k$-connectivity of $G - H$:
  \begin{enumerate}[leftmargin=2.8em, font={\bfseries\itshape}, label=(C\arabic*)]
    \item\label{C-num-edges} for all $X \subseteq V$ we have $e_G(X) \leq
      \binom{|X|}{2} p_1 + c|X| \sqrt{np_1}$,
    \item\label{C-neighbourhoods} for all $v \in V$ we have $|N_G^3(v) \cap
      \TINY| \leq 2$ and $|N_G(v) \cap \ATYP| \leq L$,
    \item\label{C-cycles} every cycle $C \subseteq G$ with $v(C) = 3$ contains
      at most one vertex from $\TINY$,
    \item\label{C-atyp-size} $|\ATYP| \leq \frac{n}{\log n}$.
  \end{enumerate}

  We show that the properties hold in $G^+$ and hence in every subgraph $G
  \subseteq G^+$. Property \ref{C-num-edges} is given by
  Lemma~\ref{lem:gnp-num-edges} applied to $G^+$ with $p_1$ (as $p$).
  Properties~\ref{C-neighbourhoods} and~\ref{C-cycles} hold by our choice of
  $\delta$ and $L$, and by Lemma~\ref{lem:neighbourhood-atyp-tiny} applied with
  $\eps/2$ (as $\eps$), $p_0$ and $p'$, since $p_0 \geq (1 + \eps/2) \log
  n/(3n)$ and $p' \leq \eps p_0$. Lastly, \ref{C-atyp-size} holds by Lemma
  \ref{lem:atypical-is-small}.

  Consider a graph $H$ on the same vertex set $V$ as $G$ which satisfies
  \eqref{eqn:H-degree-bounds}, and let $G' := G - H$. We prove the
  $k$-connectivity of $G'$ by showing that for every $S \subseteq V$ of size
  $|S| \leq k - 1$, the neighbourhood of every $X \subseteq V \setminus S$ is
  not completely contained in $X$ in the graph $G'' := G'[V \setminus S]$.
  Without loss of generality, we only consider $|X| \leq |V|/2$. Assuming
  towards a contradiction that all edges incident to the vertices belonging to
  $X$ have both endpoints in $X$ gives
  \begin{align}\label{eqn:eeq2sum}
    2 e_{G''}(X) = \sum_{v \in X} \deg_{G''}(v).
  \end{align}

  From \eqref{eqn:H-degree-bounds} we see that all vertices $v$ in $G'$ satisfy
  \begin{align}\label{eqn:min-degree-bounds}
    \deg_{G'}(v) \ge
    \begin{cases}
      \max{\{\ceil{(1/2 + \eps)\deg_G(v)}, k\}}, & \text{if } v \in \TINY, \\
      (\delta/2) np_0, & \text{if } v \in \ATYP \setminus \TINY, \\
      (1/2 + \eps/4) n p_1, & \text{otherwise},
    \end{cases}
  \end{align}
  where the last part follows from the fact that every vertex $v \in V \setminus
  \ATYP$ has degree at least $(1 - \delta) np_0$ in $G$ and thus
  \[
    \deg_{G'}(v) \geq (1/2 + \eps) \deg_G(v) \geq (1/2 + \eps)(1 - \delta)np_0
    \geq (1/2 + \eps/4)np_1,
  \]
  making use of our choice of $\delta$ and $p_1$ in the last inequality.

  We consider two cases depending on the size of $X$:
  \begin{enumerate*}[(I)]
    \item\label{I-small-x} $|X| \leq \frac{n}{1000}$ and
    \item\label{I-large-x} $|X| > \frac{n}{1000}$.
  \end{enumerate*}
  Let us look at \ref{I-small-x} first. The most critical point is to show that
  a significant part of the vertices in $X$ are actually not in $\TINY$. The
  next claim establishes precisely that.

  \begin{claim}\label{cl:x-tiny-intersection}
    $|X \cap \TINY| \leq \floor{2|X|/3}$.
  \end{claim}
  \begin{proof}
    Observe first that there cannot exist a path of length two in the induced
    subgraph $G''[X \cap \TINY]$. Indeed, if such a path exists, as it was a
    part of the giant $G$ initially, it follows that there was a vertex $v \in
    V$ with three tiny vertices (the ones on the path) in $N_G^3(v)$---a
    contradiction to \ref{C-neighbourhoods}. Consequently, the graph $G''[X \cap
    \TINY]$ consists only of isolated vertices and edges. Let $X_E$ and $X_V$ be
    the partition of $X \cap \TINY$ into a matching and an independent set and
    let
    \[
      N_E := \bigcup_{ \{u, v\} \in E(G''[X_E])} \big( N_{G''}(u) \cup N_{G''}(v)
      \big) \cap (X \setminus \TINY) \quad \text{and} \quad N_V := \bigcup_{u
      \in X_V} N_{G''}(u) \cap (X \setminus \TINY).
    \]
    If $k = 1$ then as every edge $\{u, v\} \in E(G''[X_E])$ was initially a part
    of the giant $G$, we have $|N_E| \geq |X_E|/2$ due to \ref{C-neighbourhoods}
    and \eqref{eqn:min-degree-bounds}. Additionally, if $k \geq 2$, by
    \eqref{eqn:min-degree-bounds} together with properties
    \ref{C-neighbourhoods} and \ref{C-cycles} one easily derives (assuming
    $|X_E| > 0$, otherwise $|N_E| \geq |X_E|$ trivially holds)
    \[
      |N_E| \geq (2k - 2)|X_E| - |S| \geq |X_E| + (2k - 3)|X_E| - k + 1 \geq
      |X_E|.
    \]
    Similarly, as no three vertices from $X_V$ can have a common neighbour in $X
    \setminus \TINY$ by \ref{C-neighbourhoods}, we deduce $|N_V| \geq
    \ceil{|X_V|/2}$. Indeed, every vertex has at least $k$ neighbours in $X
    \setminus \TINY$ and at most $k - 1$ are removed due to the removal of $S$;
    moreover, at most one other vertex in $X_V$ can share this remaining
    neighbour due to \ref{C-neighbourhoods}.

    Clearly now, if the number of vertices in $X \setminus \TINY$ is strictly
    smaller than $|X_E|/2 + \ceil{|X_V|/2}$ by the pigeonhole principle there is
    a vertex $v \in X \setminus \TINY$ which has at least three vertices from $X
    \cap \TINY$ in $N_{G''}^3(v)$, again a contradiction to
    \ref{C-neighbourhoods}. The claim readily follows.
  \end{proof}

  The remaining vertices can be partitioned into $X \setminus \TINY = \Xa \cup
  \Xt$ where
  \[
    \Xa := (\ATYP \setminus \TINY) \cap X \qquad \text{and} \qquad \Xt := X
    \setminus \ATYP.
  \]
  Next, we show that $|\Xt| > |\Xa|$. Assume towards contradiction that this is
  not the case. As $np_0 = \omega(1)$ and $|S| \leq k - 1$ we have
  \begin{align*}
    e_{G''}(\Xa, \Xt) &\osref{\eqref{eqn:min-degree-bounds}}\geq \big(
    (\delta/2)np_0 - (k - 1) \big) \cdot |\Xa| - 2 e_{G''}(\Xa) - e_{G''}(\Xa,
    \TINY) \\
    &\osref{\ref{C-neighbourhoods}}\geq (\delta/4) np_0 \cdot |\Xa| - 2L \cdot
    |\Xa| - 2 \cdot |\Xa| > L |\Xt|,
  \end{align*}
  where the last inequality follows from our assumption $|\Xa| \geq |\Xt|$.
  Thus, a simple averaging argument shows that there exists a vertex $u \in \Xt$
  with $\deg_{G''}(u, \Xa) > L$, which is a contradiction to
  \ref{C-neighbourhoods}. Together with Claim~\ref{cl:x-tiny-intersection} we
  conclude $|\Xt| \geq |X|/6$. Therefore, we have
  \[
    \sum_{v \in X} \deg_{G''}(v) \ge |\Xt| \cdot \frac{1}{2} n p_1 \geq
    \frac{|X|np_1}{12},
  \]
  where we again make use of \eqref{eqn:min-degree-bounds}. On the other hand,
  by \ref{C-num-edges} we also have $2e_{G''}(X) \leq |X|^2p_1 +
  2c|X|\sqrt{np_1}$, and as $|X| \leq n/1000$ the assumption \eqref{eqn:eeq2sum}
  is wrong. Thus, $N_{G''}(X) \setminus X \neq \varnothing$.

  We now consider case \ref{I-large-x}. By \ref{C-atyp-size}, at most $n/\log n$
  vertices are in $\ATYP$, which implies that at least $(1 - \eps/8) |X|$
  vertices are in $X \setminus \ATYP$, as $|X| > n/1000$. Thus, by
  \eqref{eqn:min-degree-bounds} the sum of the degrees of the vertices in $X$ in
  $G''$ can be bounded by
  \[
    \sum_{v \in X} \deg_{G''}(v) \geq (1 - \eps/8)|X| \cdot (1/2 + \eps/4)np_1.
  \]
  Combining \ref{C-num-edges} and the bound from above with
  equation~\eqref{eqn:eeq2sum} gives
  \[
    (1 - \eps/8)(1/2 + \eps/4)|X| np_1 \leq 2 \binom{|X|}{2} p_1 +
    2c|X|\sqrt{np_1} \leq (1/2 + \eps/16)|X|np_1,
  \]
  where the last inequality follows from the fact that $np_1 = \omega(1)$ and
  $|X| \leq n/2$---again a contradiction.

  In conclusion, there is no connected component of $G''$ of size at most
  $|V|/2$, which completes the proof of the proposition.
\end{proof}

Having Proposition~\ref{prop:graph-process} at hand we proceed to complete the
proof of the main result. For convenience of the reader we restate the theorem.

\conntheorem*
\begin{proof}
  With a slightly more careful inspection, one can see that the conclusion of
  Proposition~\ref{prop:graph-process} holds with probability $1 - e^{-\alpha
  \cdot C\log n}$ for every fixed $m \geq C n\log n$, some $\alpha > 0$, and for
  sufficiently large $C$ depending on $\eps$. Therefore, a union bound over all
  $m \geq C n\log n$ shows that with probability at least $1 - n^2 \cdot e^{-
  \alpha C\log n} = 1 - o(1)$, every $G_m$ is resilient with respect to
  connectivity. Note that it is also possible to draw the same conclusion from
  the results of Montgomery~\cite{montgomery2017} and Nenadov, Steger, and the
  second author~\cite{nenadov2018resilience} as connectivity is a necessary
  condition for Hamiltonicity (even starting at $m \geq \frac{1 + \eps}{2}n \log
  n$).

  For the smaller values of $m$, consider intervals of the form $[\frac{1 + i
  \eps}{6} n \log n, \frac{1 + (i + 1) \eps}{6} n \log n)$, for $i \in \{1,
  \ldots, C_\eps\}$, where $C_\eps$ is such that the last interval contains $Cn
  \log n$. For each fixed interval the conclusion of
  Proposition~\ref{prop:graph-process} holds with probability $1 - o(1)$, thus a
  union bound over constantly many intervals shows that it w.h.p.\ holds for all
  intervals simultaneously. This concludes the proof.
\end{proof}

The proof of Theorem~\ref{thm:main-thm-k-conn} is analogous. This immediately
implies Theorem~\ref{thm:hitting-time-conn} and
Theorem~\ref{thm:hitting-time-k-conn}.

We conclude by showing that if the adversary is allowed to remove slightly more
than a $(1/2)$-fraction of the edges incident to each vertex, then w.h.p.\ it is
possible to make the giant of $G_m$ disconnected. For simplicity we only show an
analogue of Proposition~\ref{prop:graph-process}. In order to capture all values
of $m$ one proceeds as in the proof of Theorem~\ref{thm:main-thm-conn}.

\begin{proposition}\label{prop:optimal-constant}
  For every $\eps > 0$ and integer $m_0 \geq \frac{1 + \eps}{6} n\log n$ the
  random graph process $\{G_i\}$ w.h.p.\ has the following property. For every
  integer $m_0 \leq m \leq (1 + \eps/6)m_0$, the giant of $G_m$ is not $(1/2 +
  \eps)$-resilient with respect to connectivity.
\end{proposition}
\begin{proof}
  Given $\eps$, let $\delta = \min{\{\eps/32,
  \delta_{\ref{lem:neighbourhood-atyp-tiny}}(\eps/2)\}}$, and $L =
  L_{\ref{lem:neighbourhood-atyp-tiny}}(\eps/2)$. Take $p_0 = (1 - \eps/16) m_0
  / \binom{n}{2}$, $p' = (\eps/2) p_0$, and let $G^+$ be the union of two
  independent copies of random graphs $G^- \sim G_{n, p_0}$ and $G_{n, p'}$.
  Then $G^+$ is distributed as $G_{n, p_1}$, where $p_1 = 1 - (1 - p_0)(1 -
  p')$. Similarly as before, we have that the number of edges in $G^-$ is
  w.h.p.\ at most $m_0$ and the number of edges in $G^+$ is w.h.p.\ at least $(1
  + \eps/6) m_0$. For the rest of the proof consider some $m_0 \leq m \leq (1 +
  \eps/6) m_0$.

  Let $G$ be the giant of $G_m$ and $V$ its vertex set. Suppose there exists a
  partition $V = A \cup B$ such that the following property is satisfied:
  \begin{equation}\label{eq:deg-condition}
    \text{for all } v \in V \colon \deg_{G[A, B]}(v) \leq (1/2 + \eps)\deg_G(v).
    \tag{$\star$}
  \end{equation}
  Then by taking $H$ to be a subgraph consisting of all edges in $G[A, B]$ the
  graph $G - H$ is not connected. In the remainder we show that such a partition
  indeed exists.

  Consider first an arbitrary equipartition $V = A' \cup B'$ (i.e.\ $A'$ and
  $B'$ differ by at most one in size). As before, let
  \[
    \TINY := \TINY_{p_0, \delta}(G^-) \cap V \qquad \text{and} \qquad \ATYP :=
    \ATYP_{p_0, \delta}(G^-) \cap V.
  \]
  and assume that the property \ref{C-neighbourhoods} holds, which follows from
  Lemma~\ref{lem:neighbourhood-atyp-tiny} applied with $\eps/2$ (as $\eps$). Let
  $D$ be defined as:
  \[
    D := \{ v \in V \colon \deg_{G[A', B']}(v) > (1/2 + \delta)np_1 \}.
  \]

  \begin{claim}\label{cl:degenerate-neighbourhoods}
    There exists a positive constant $L'(\eps)$ such that w.h.p.\ for all $v \in
    V$ we have $|N_{G^+}(v) \cap D| < L'$.
  \end{claim}
  The proof follows an analogous argument as the proof of
  \cite{nenadov2018resilience}*{Lemma 2.7}.
  \begin{proof}
    We show that, for a sufficiently large constant $L' = L'(\eps)$, if $T
    \subseteq G^+$ is a tree with $L' \leq v(T) \leq 2L'$ vertices then it
    contains at most $L' - 1$ vertices from $D$, which is sufficient for the
    claim to hold. Clearly, if for a vertex $v \in V$ there are $L'$ vertices in
    $N_{G^+}(v) \cap D$ then taking a tree containing $v$ and every such vertex
    $u$ together with the edge $\{v, u\}$ contradicts the former.

    Let $T \subseteq K_n$ be a tree with $L' \leq v(T) \leq 2L'$ and $S
    \subseteq V(T)$ a set of size exactly $L'$. Let $\cE_T$ denote the event
    that $T \subseteq G^+$ and $\cE_{T, S}$ that every $v \in S$ satisfies
    $|N_{G^+[A', B']}(v) \setminus V(T)| > (1/2 + \delta/2)np_1$. By Chernoff's
    inequality we have that a fixed vertex $v \in S$ satisfies this with
    probability at most
    \[
      \Pr \big[ \Bin\big( \tfrac{n}{2}, p_1 \big) > (1/2 + \delta/2)np_1 \big]
      \leq e^{- \gamma np_1},
    \]
    for some $\gamma > 0$ depending on $\delta$ (and thus $\eps$). As these
    events are independent for different vertices $v, u \in S$, the probability
    that $\cE_{T, S}$ holds is bounded by
    \[
      \Pr[\cE_{T, S}] \leq (e^{- \gamma np_1})^{L'}.
    \]
    Note that $\Pr[\cE_T] = p_1^{v(T) - 1}$. Hence, a simple union bound over
    all pairs $(T, S)$ shows that the probability that some $\cE_T \land \cE_{T,
    S}$ happens is at most
    \[
      \Pr \Big[ \bigcup_{(T, S)} (\cE_T \land \cE_{T, S}) \Big] \leq \sum_{t =
      L'}^{2L'} \binom{n}{t} \binom{t}{L'} t^{t - 2} \cdot \Pr[\cE_T \land
      \cE_{T, S}] \leq \sum_{t = L'}^{2L'} n^t t^{3L'} \cdot p_1^{t - 1} \cdot
      e^{- L'\gamma np_1}.
    \]
    As $np_1 = \Omega(\log n)$ and $n^t p_1^{t - 1} \cdot e^{-L'\gamma np_1}$ is
    decreasing in $p_1$, this finally implies
    \[
      \Pr \Big[ \bigcup_{(T, S)} (\cE_T \land \cE_{T, S}) \Big] = O_{\eps}\big(
      n \cdot (\log n)^{2L'} \cdot e^{-L'\gamma np_1} \big) = o(1),
    \]
    for $L'$ large enough depending on $\eps$.
  \end{proof}

  We now construct the partition $V = A \cup B$ as follows. In the beginning set
  $A_0 := A' \setminus (\ATYP \cup D)$ and $B_0 := B' \setminus (\ATYP \cup D)$.
  We sequentially add vertices to $A_0$ and $B_0$, first those of $(\ATYP \cup D)
  \setminus \TINY$ and then of $\TINY$, in an arbitrary order following a simple
  rule: if in step $i \geq 1$ we have $\deg_G(v, A_{i - 1}) \geq \deg_G(v, B_{i
  - 1})$ set $A_i := A_{i - 1} \cup \{v\}$ and $B_i := B_{i - 1}$; otherwise set
  $A_i := A_{i - 1}$ and $B_i := B_{i - 1} \cup \{v\}$. Lastly, set $A := A_m$
  and $B := B_m$ for $m := |\ATYP \cup D|$.

  Note that the degree of a vertex added at step $i \geq 0$ in graph $G[A_i,
  B_i]$ can increase by at most $L + L' + 2$ until the end of the process by
  \ref{C-neighbourhoods} and Claim~\ref{cl:degenerate-neighbourhoods}.
  Therefore, for every vertex $v \in V \setminus (\ATYP \cup D)$ we have
  \[
    \deg_{G[A, B]}(v) \leq (1/2 + \delta)np_1 + L + L' + 2 \leq (1/2 +
    2\delta)np_1 \leq (1/2 + \eps/2)\deg_G(v),
  \]
  as $\deg_G(v) \geq (1 - \delta)np_0$, $p_1 \leq (1 + \eps/2)p_0$, and due to
  our choice of $\delta$. Similarly, for $v \in (\ATYP \cup D) \setminus \TINY$
  we have
  \[
    \deg_{G[A, B]}(v) \leq \deg_G(v)/2 + L + L' + 2 \leq (1/2 +
    \eps/2)\deg_G(v),
  \]
  as $\deg_G(v) \geq \delta np_0$ and due to our choice of $\delta$.

  Lastly, let us look at tiny vertices. If after the end of the process there is
  no vertex $v \in \TINY$ which has $\deg_{G[A, B]}(v) > \deg_G(v)/2$, we are
  done. Hence, assume w.l.o.g.\ that $v \in A$ is such a vertex. In particular,
  this implies that $v$ had a neighbour $u \in \TINY$ added to $B$ after $v$
  itself. Since $u \in B$, it implies that there exists $w \in N_G(u, B)$.
  Consequently, due to \ref{C-neighbourhoods}, none of the neighbours of $v$ in
  $A$ belong to $\TINY$ and we may move $v$ to $B$ without harming the degree of
  any vertex significantly (it changes by at most one for non-tiny vertices,
  which is negligible compared to their degree). Moreover, by doing this we
  cannot make additional vertices $v' \in \TINY$ have $\deg_{G[A, B]}(v') >
  \deg_G(v')/2$ and thus the rearranging process eventually stops with all
  vertices satisfying $\eqref{eq:deg-condition}$. This completes the proof.
\end{proof}

\paragraph*{Acknowledgements.} We would like to thank the anonymous reviewers
for their thorough reading of our paper and helpful comments.

\bibliographystyle{abbrv}
\bibliography{references}

\begin{bibdiv}
\begin{biblist}

\bib{allen2016bandwidth}{article}{
      author={Allen, Peter},
      author={B{\"o}ttcher, Julia},
      author={Ehrenm{\"u}ller, Julia},
      author={Taraz, Anusch},
       title={The bandwidth theorem in sparse graphs},
        date={2016},
     journal={arXiv preprint arXiv:1612.00661},
}

\bib{balogh2011local}{article}{
      author={Balogh, J{\'o}zsef},
      author={Csaba, B{\'e}la},
      author={Samotij, Wojciech},
       title={Local resilience of almost spanning trees in random graphs},
        date={2011},
     journal={Random Structures \& Algorithms},
      volume={38},
      number={1-2},
       pages={121\ndash 139},
}

\bib{balogh2012corradi}{article}{
      author={Balogh, J{\'o}zsef},
      author={Lee, Choongbum},
      author={Samotij, Wojciech},
       title={Corr{\'a}di and {H}ajnal's theorem for sparse random graphs},
        date={2012},
     journal={Combinatorics, Probability and Computing},
      volume={21},
      number={1-2},
       pages={23\ndash 55},
}

\bib{bollobas84evolutionhitting}{incollection}{
      author={Bollob\'as, B\'ela},
       title={The evolution of sparse graphs},
        date={1984},
   booktitle={Graph theory and combinatorics ({C}ambridge, 1983)},
   publisher={Academic Press, London},
       pages={35\ndash 57},
      review={\MR{777163}},
}

\bib{bollobas1998random}{incollection}{
      author={Bollob{\'a}s, B{\'e}la},
       title={Random graphs},
        date={1998},
   booktitle={Modern graph theory},
   publisher={Springer},
       pages={215\ndash 252},
}

\bib{bollobas1985random}{article}{
      author={Bollob{\'a}s, B{\'e}la},
      author={Thomason, Andrew},
       title={Random graphs of small order},
        date={1985},
     journal={North-Holland Mathematics Studies},
      volume={118},
       pages={47\ndash 97},
}

\bib{bondy2008graphtheory}{book}{
      author={Bondy, J.~A.},
      author={Murty, U. S.~R.},
       title={Graph theory},
      series={Graduate Texts in Mathematics},
   publisher={Springer, New York},
        date={2008},
      volume={244},
}

\bib{bottcher2017large}{article}{
      author={B{\"o}ttcher, Julia},
       title={Large-scale structures in random graphs},
        date={2017},
     journal={Surveys in Combinatorics 2017},
      volume={440},
       pages={87},
}

\bib{bottcher2013almost}{article}{
      author={B{\"o}ttcher, Julia},
      author={Kohayakawa, Yoshiharu},
      author={Taraz, Anusch},
       title={Almost spanning subgraphs of random graphs after adversarial edge
  removal},
        date={2013},
     journal={Combinatorics, Probability and Computing},
      volume={22},
      number={5},
       pages={639\ndash 683},
}

\bib{bottcher2009proof}{article}{
      author={B{\"o}ttcher, Julia},
      author={Schacht, Mathias},
      author={Taraz, Anusch},
       title={Proof of the bandwidth conjecture of {B}ollob{\'a}s and
  {K}oml{\'o}s},
        date={2009},
     journal={Mathematische Annalen},
      volume={343},
      number={1},
       pages={175\ndash 205},
}

\bib{dirac1952some}{article}{
      author={Dirac, Gabriel~Andrew},
       title={Some theorems on abstract graphs},
        date={1952},
     journal={Proceedings of the London Mathematical Society},
      volume={3},
      number={1},
       pages={69\ndash 81},
}

\bib{erdds1959random}{article}{
      author={Erd{\H{o}}s, P},
      author={R{\'e}nyi, A},
       title={On random graphs {I}},
        date={1959},
     journal={Publicationes Mathematicae Debrecen},
      volume={6},
       pages={290\ndash 297},
}

\bib{erdHos1964strength}{article}{
      author={Erd{\H{o}}s, Paul},
      author={R{\'e}nyi, Alfr{\'e}d},
       title={On the strength of connectedness of a random graph},
        date={1964},
     journal={Acta Mathematica Academiae Scientiarum Hungarica},
      volume={12},
      number={1-2},
       pages={261\ndash 267},
}

\bib{frieze2015introduction}{book}{
      author={Frieze, Alan},
      author={Karo{\'n}ski, Micha{\l}},
       title={Introduction to random graphs},
   publisher={Cambridge University Press},
        date={2015},
}

\bib{gilbert1959random}{article}{
      author={Gilbert, Edgar~N},
       title={Random graphs},
        date={1959},
     journal={The Annals of Mathematical Statistics},
      volume={30},
      number={4},
       pages={1141\ndash 1144},
}

\bib{hajnal1970proof}{article}{
      author={Hajnal, Andr{\'a}s},
      author={Szemer{\'e}di, Endre},
       title={Proof of a conjecture of {P.} {E}rd{\H{o}}s},
        date={1970},
     journal={Combinatorial theory and its applications},
      volume={2},
       pages={601\ndash 623},
}

\bib{komlos2001proof}{article}{
      author={Koml{\'o}s, J{\'a}nos},
      author={S{\'a}rk{\"o}zy, G{\'a}bor},
      author={Szemer{\'e}di, Endre},
       title={Proof of the {A}lon-{Y}uster conjecture},
        date={2001},
     journal={Discrete Mathematics},
      volume={235},
      number={1-3},
       pages={255\ndash 269},
}

\bib{komlos1998posa}{article}{
      author={Koml{\'o}s, J{\'a}nos},
      author={S{\'a}rk{\"o}zy, G{\'a}bor~N},
      author={Szemer{\'e}di, Endre},
       title={On the {P}{\'o}sa-{S}eymour conjecture},
        date={1998},
     journal={Journal of Graph Theory},
      volume={29},
      number={3},
       pages={167\ndash 176},
}

\bib{kouider1996covering}{article}{
      author={Kouider, Mekkia},
      author={Lonc, Zbigniew},
       title={Covering cycles and $k$-term degree sums},
        date={1996},
     journal={Combinatorica},
      volume={16},
      number={3},
       pages={407\ndash 412},
}

\bib{krivelevich2006pseudo}{article}{
      author={Krivelevich, Michael},
      author={Sudakov, Benny},
       title={Pseudo-random graphs},
        date={2006},
     journal={More sets, graphs and numbers},
       pages={199\ndash 262},
}

\bib{lee2012dirac}{article}{
      author={Lee, Choongbum},
      author={Sudakov, Benny},
       title={Dirac's theorem for random graphs},
        date={2012},
     journal={Random Structures \& Algorithms},
      volume={41},
      number={3},
       pages={293\ndash 305},
}

\bib{luczak1989sparse}{inproceedings}{
      author={{\L}uczak, Tomasz},
       title={Sparse random graphs with a given degree sequence},
        date={1989},
   booktitle={Proceedings of the symposium on random graphs, {P}oznan},
       pages={165\ndash 182},
}

\bib{luczak1991size}{article}{
      author={{\L}uczak, Tomasz},
       title={Size and connectivity of the $k$-core of a random graph},
        date={1991},
     journal={Discrete Mathematics},
      volume={91},
      number={1},
       pages={61\ndash 68},
}

\bib{luczak1991tree}{article}{
      author={{\L}uczak, Tomasz},
      author={Ruci{\'n}ski, Andrzej},
       title={Tree-matchings in graph processes},
        date={1991},
     journal={SIAM Journal on Discrete Mathematics},
      volume={4},
      number={1},
       pages={107\ndash 120},
}

\bib{montgomery2017}{article}{
      author={Montgomery, Richard},
       title={Hamiltonicity in random graphs is born resilient},
        date={2017},
     journal={arXiv preprint arXiv:1710.00505},
}

\bib{nenadov2018resilience}{article}{
      author={Nenadov, Rajko},
      author={Steger, Angelika},
      author={Truji{\'c}, Milo{\v{s}}},
       title={Resilience of perfect matchings and {H}amiltonicity in random
  graph processes},
        date={2018},
     journal={Random Structures \& Algorithms},
      volume={0},
      number={0},
}

\bib{vskoric2018local}{article}{
      author={{\v{S}}kori{\'c}, Nemanja},
      author={Steger, Angelika},
      author={Truji{\'c}, Milo{\v{s}}},
       title={Local resilience of an almost spanning k-cycle in random graphs},
        date={2018},
     journal={Random Structures \& Algorithms},
      volume={53},
      number={4},
       pages={728\ndash 751},
}

\bib{sudakov2017robustness}{article}{
      author={Sudakov, Benny},
       title={Robustness of graph properties},
        date={2017},
     journal={Surveys in Combinatorics 2017},
      volume={440},
       pages={372},
}

\bib{sudakov2008local}{article}{
      author={Sudakov, Benny},
      author={Vu, Van~H},
       title={Local resilience of graphs},
        date={2008},
     journal={Random Structures \& Algorithms},
      volume={33},
      number={4},
       pages={409\ndash 433},
}

\bib{turan1941extremalaufgabe}{article}{
      author={Tur{\'a}n, Paul},
       title={Eine extremalaufgabe aus der graphentheorie},
        date={1941},
     journal={Mat. Fiz. Lapok},
      volume={48},
      number={436-452},
       pages={61},
}

\end{biblist}
\end{bibdiv}

\end{document}